\documentclass{amsart}
\usepackage{amssymb,latexsym}
\usepackage{amscd,amsthm}
\usepackage{mathrsfs}
\usepackage{marginnote}
\usepackage[usenames,dvipsnames,svgnames,table]{xcolor}
\usepackage{inputenc}

\usepackage[all]{xy}
\usepackage{tikz}

\usepackage{tikz-cd}

\usepackage[linkbordercolor=BrickRed, citebordercolor=OliveGreen]{hyperref}

\newtheorem{theorem}{Theorem}[section]

\newtheorem{lemma}[theorem]{Lemma}

\newtheorem{proposition}[theorem]{Proposition}

\theoremstyle{definition}
\newtheorem{definition}[theorem]{Definition}
\newtheorem{fact}[theorem]{Fact}
\newtheorem{remark}[theorem]{Remark}

\newtheorem{example}[theorem]{Example}

\newtheorem{ipg}[theorem]{}

\DeclareMathOperator{\Ext}{Ext}
\DeclareMathOperator{\Hom}{Hom}

\DeclareMathOperator{\Ho}{Ho}

\DeclareMathOperator{\rep}{Rep}
\DeclareMathOperator{\Proj}{Proj}
\DeclareMathOperator{\Inj}{Inj}
\DeclareMathOperator{\FPI}{FPI}

\DeclareMathOperator{\GProj}{GProj}
\DeclareMathOperator{\GInj}{GInj}

\DeclareMathOperator{\DProj}{DProj}
\DeclareMathOperator{\DInj}{DInj}

\newcommand{\rightperp}[1]{#1^{\perp}}
\newcommand{\leftperp}[1]{{}^\perp #1}

\newcommand{\class}[1]{\mathcal{#1}}   

\newcommand{\coker}{\textnormal{coker}\mspace{3mu}}

\newcommand{\sGProj}{\underline{\mathrm{GPro}}\mathrm{j}}
\newcommand{\sGInj}{\underline{\mathrm{GIn}}\mathrm{j}}
\newcommand{\sDProj}{\underline{\mathrm{DPro}}\mathrm{j}}
\newcommand{\sDInj}{\underline{\mathrm{DIn}}\mathrm{j}}

\begin{document}
\emergencystretch 3em

\title{Abelian model structures on categories of quiver representations}

\author{Georgios Dalezios}
\address{Departamento de Matem\'aticas, Universidad de Murcia, 30100 Murcia, Spain}
\email{georgios.dalezios@um.es}
\address{Department of Mathematical Sciences, University of Copenhagen, Universitets\-parken 5, 2100 Copenhagen {\O}, Denmark}

\thanks{The author is supported by the Fundaci\'{o}n S\'{e}neca of Murcia 19880/GERM/15.}
\date{}



\begin{abstract}
Let $\class  M$ be an abelian model category (in the sense of Hovey). For a large class of quivers, we describe associated abelian model structures on categories of quiver representations with values in $\class M$. This is based on recent work of Holm and J\o rgensen on cotorsion pairs in categories of quiver representations. An application on Ding projective and Ding injective representations of quivers over Ding-Chen rings is given.
\end{abstract}

\maketitle

\section{Introduction}
Model structures on abelian categories have been studied extensively by Hovey \cite{hovey}, who introduced the general notion of an abelian model structure on an abelian category $\class M$, and gave a correspondence between such models and certain cotorsion pairs in $\class M$. A cotorsion pair in an abelian category $\class M$ is a pair of $\Ext_{\class M}^{1}(-,-)$--orthogonal to each other subcategories. The basic idea behind Hovey's results is that, for an abelian model category $\class M$, the various lifting properties in the model $\class M$ can be interpreted as certain $\Ext_{\class M}^{1}(-,-)$--orthogonality relations. Thus in order to give a model structure on an abelian category, it suffices to find certain cotorsion pairs and then use the correspondence of Hovey.

Given an abelian model structure on an abelian category $\class M$ and a quiver (a directed graph) $Q$, we consider the category of quiver representations $\rep_{Q}\class M$, that is, diagrams of shape $Q$ in $\class M$, and study how the given model on $\class M$ transfers to the abelian category $\rep_{Q}\class M$. Representations of quivers in module categories are of interest in the representation theory of finite dimensional algebras \cite{rta}. Moreover, derived categories of the category $\rep_{Q}\class M$ are usually thought of as enhancements of the derived category of $\class M$ and have recently attracted much attention, see for instance \cite{MR3513838, MR3474336}. 

In general, for a given model category $\class M$ and a small category $I$, a model on the functor category $\class M^{I}$ might exist or not, depending on conditions on either $\class M$ or $I$, see \cite[Chapters 11, 15]{MR1944041}. In Theorems \ref{Thm-proj-model}/\ref{Thm-inj-model} we give a description of certain projective (resp. injective) model structures on categories of quiver representations, based on cotorsion pairs in such categories as obtained by Holm and J\o rgensen \cite{Holm-Jorgensen2016}. The examples we are interested in here are of a ring-theoretic flavour.
In \ref{ex-Gproj}/\ref{ex-GInj} we provide examples which realize stable categories of Gorenstein projective (resp. injective) representations of left (resp. right) rooted quivers, over certain rings, as Quillen homotopy categories. 

The last section is concerned with quiver representations over Ding-Chen rings, a generalization of Gorenstein rings studied by Gillespie \cite{MR2607410}. In Theorems \ref{ding-proj-model}/\ref{dinj-model} we provide abelian model structures for Ding projective and Ding injective representations over such rings, which generalize the analogous statements for Gorenstein rings from \ref{ex-Gproj}/\ref{ex-GInj}.

\section{Preliminaries}

In this section we briefly summarize some known facts on cotorsion pairs, abelian model structures and quiver representations.

\begin{ipg}\textbf{Cotorsion pairs.}
Let $\class M$ be an abelian category. For a class $\class C$ of objects in $\class M$ the right orthogonal $\rightperp{\class C}$ is defined to be the class of all $M \in \class M$ such that $\Ext^1_{\class M}(C,M) = 0$ for all $C \in \class C$. The left orthogonal $\leftperp{\class{C}}$ is defined analogously.
We say that a pair $(\class X,\class Y$) of classes of objects in $\class M$ is a \textit{cotorsion pair} if $\class X=\leftperp{\class{Y}}$ and $\class Y=\rightperp{\class{X}}$. 
A cotorsion pair $(\class X,\class Y)$ is called \textit{complete} if for every object $M$ in $\class M$ there exists a short exact sequence 
$0\rightarrow Y\rightarrow X\rightarrow M\rightarrow 0$ with $X\in\class X$ and $Y\in\class Y$, and also a short exact sequence 
$0\rightarrow M\rightarrow Y'\rightarrow X'\rightarrow 0$ with $X'\in\class X$  and $Y'\in\class Y$. It is called \textit{hereditary} if for all $X\in\class X, Y\in\class Y$ and $i\geq 1$,\,  $\Ext^{i}_{\class M}(X,Y)=0$. We refer to \cite{GobelTrlifaj} for the theory of cotorsion pairs.
Following the terminology of \cite[Dfn.~2.2.1]{GobelTrlifaj}, a cotorsion pair is said to be \textit{generated}, respectively \textit{cogenerated}, by a set of objects $\mathcal{S}$, if it is of the form $(\leftperp{(\rightperp{\mathcal{S})}},\rightperp{\mathcal{S}})$, respectively $(\leftperp{\class S},\rightperp{(\leftperp{\class S})})$.
\end{ipg}

\begin{ipg}\textbf{Abelian model structures.}
\label{h1}
Let $\class M$ be an abelian category. Following Hovey \cite[Dfn.~2.1]{hovey} we say that $\class M$ admits an \textit{abelian model structure} (or that $\class M$ is an abelian model category), if it admits a Quillen model structure \cite{modcat} where the (trivial) cofibrations are the monomorphisms with (trivially) cofibrant kernel, and the (trivial) fibrations are the epimorphisms with (trivially) fibrant kernel. If we denote by $\class C, \class F$ and $\class W$ the classes of cofibrant, fibrant and trivial (i.e.~weakly isomorphic to zero) objects in this model category, we obtain from \cite[Thm.~2.2]{hovey} two (functorially) complete cotorsion pairs $(\class C\cap\class W,\class F)$,  $(\class C,\class W\cap\class F)$ in the category $\class M$. Conversely, for classes of objects $\class C,\class W$ and $\class F$, where $\class W$ is thick, any two complete cotorsion pairs of the above form give rise to an abelian model structure on $\class M$, see again \cite[Thm.~2.2]{hovey}. We abbreviate by saying that $(\class C,\class W,\class F)$ is a \textit{Hovey triple} on the category $\class M$. 
\end{ipg}

\begin{ipg}
\label{h2}
A Hovey triple $(\class C,\class W,\class F)$ on an abelian category $\class M$ is called \textit{hereditary} if the corresponding complete cotorsion pairs $(\class C\cap\class W,\class F)$ and $(\class C,\class W\cap\class F)$ are hereditary. In this case, the category $\class C\cap\class F$ is Frobenius \cite[Prop.~5.2.(4)]{Gil2011} (where $\class C \cap \class F\cap\class W$ is the class projective--injective objects) and the homotopy category of the model category $\class M$ (in the classical sense) is canonically equivalent to the stable category of the Frobenius category $\class C\cap\class F$, we refer to \cite[Section~4.2]{Gil2011} for the details.
\end{ipg}

\begin{ipg}\textbf{Setup.}
\label{setup}
Throughout the text $\class M$ denotes an abelian category with enough projectives and injectives which satisfies the axioms AB4 and AB4*, that is, $\class M$ is bicomplete and such that any coproduct of monomorphisms in $\class M$ is a monomorphism, and dually any product of epimorphisms in $\class M$ is an epimorphism.
\end{ipg}

\begin{ipg}
\label{quiver}\textbf{Quivers.}
We recall that a quiver $Q=(Q_{0},Q_{1})$ is a directed graph $Q$ with set of vertices $Q_{0}$ and set of arrows $Q_{1}$. For $\alpha\in Q_{1}$ we denote by $s(\alpha)$ its source and by $t(\alpha)$ its target.
If $Q$ is a quiver and $\class X$ is a class of objects in $\class M$, then viewing $Q$ as a small category, we consider the category $\rep_{Q}\class X$ of diagrams of shape $Q$ in $\class X$. The objects of $\rep_{Q}\class X$ are also called $\class X$-valued representations of $Q$. For any such representation $X$ and any vertex $i\in Q_{0}$, there exist two natural maps
\[\bigoplus\limits_{\alpha:j\rightarrow i}X(j)\xrightarrow{\phi_{i}^{X}} X(i)\,\,\,\,\,\,\,\,\,\, \mbox{and}\,\,\,\,\,\,\,\,\,\, X(i)\xrightarrow{\psi_{i}^{X}} \prod\limits_{\alpha:i\rightarrow j}X(j).\]

For a quiver $Q$, consider, as in \cite[Section~4]{EEGR-injective-quivers}, a sequence of subsets of $Q_{0}$ defined by transfinite recursion as follows: 
Put $W_{0}:=\emptyset$, for a successor $\alpha=\beta+1$, put
\[W_{\alpha}:=\{i\in Q_{0}\, |\, i\,\, \mbox{is not the source of any arrow that has target outside of}\,\,  W_{\beta}\}\]
and for a limit ordinal $\alpha$ put $W_{\alpha}:=\cup_{\beta<\alpha}W_{\beta}$.

A quiver is called \textit{right rooted} if for some ordinal $\lambda$ we have $W_{\lambda}=Q_{0}$. From \cite[Section~4]{EEGR-injective-quivers} we have that $Q$ is right rooted if and only if it does not contain any path of the form $\bullet\rightarrow\bullet\rightarrow\bullet\rightarrow\cdots$. A dual definition and a (dual) characterization holds for \textit{left rooted} quivers, see \cite{MR2100360}.
\end{ipg}

\begin{ipg} 
\label{adjoints}\textbf{Adjoints of evaluation functors.}
Let $Q$ be a quiver, $i\in Q_{0}$ a vertex and let $\class A$ be a category that admits finite products and finite coproducts. We recall, for instance from \cite[3.7]{Holm-Jorgensen2016}, that the evaluation at $i$ functor $(-)(i):\rep_{Q}\class A\rightarrow\class A;\, X\mapsto X(i)$, admits a left adjoint $f_{i}$ and a right adjoint $g_{i}$ which are defined, on a vertex $j$, by the rules $f_{i}(M)(j):=\coprod\limits_{\alpha:i\rightarrow j} M$ and $g_{i}(M)(j)=\prod\limits_{\alpha:j\rightarrow i} M$ respectively. We refer to \cite[Section~3]{Holm-Jorgensen2016} for the full definition and properties of these functors.
\end{ipg}

We start by recalling some of the main results of \cite{Holm-Jorgensen2016}.

\begin{fact} \cite[Thm.~A]{Holm-Jorgensen2016}
\label{fact1}
Let $Q$ be a left rooted quiver and let $\class M$ be abelian category as in setup \ref{setup}. If $(\class A,\class B)$ is a cotorsion pair in $\class M$, then there is an induced cotorsion pair $(\Phi(\class A), \rep_{Q}B)$ in the category $\rep_{Q}\class M$, where 
\[\Phi(\class A):=\{X\, |\,\, \forall i\in Q_{0},\,\, \phi_{i}^{X}\,\, \mbox{is monic with}\,\, X(i)\in\class A,\,\, \coker\phi_{i}^{X}\in\class A\}.\]
In addition, if $(\class A,\class B)$ is hereditary or generated by a set, then so is $(\Phi(\class A), \rep_{Q}\class B)$.
\end{fact}

The dual of this statement is as follows: 

\begin{fact}\cite[Thm.~B]{Holm-Jorgensen2016}
\label{fact2}
Let $Q$ be a right rooted quiver and let $\class M$ be abelian category as in setup \ref{setup}. If $(\class A,\class B)$ is a cotorsion pair in $\class M$, then there is an induced cotorsion pair $(\rep_{Q}\class A,\Psi(\class B))$ in the category $\rep_{Q}\class M$, where
\[\Psi(\class B):=\{X\, |\,\, \forall i\in Q_{0},\,\, \psi_{i}^{X}\,\, \mbox{is epic with}\,\,  X(i)\in\class B,\,\, \ker\psi_{i}^{X}\in\class B\}.\]
In addition, if $(\class A,\class B)$ is hereditary or generated by a set, then so is $(\rep_{Q}\class A,\Psi(\class B))$.
\end{fact}

The following is the main result of \cite{odabasi-completeness} and addresses the question of when the cotorsion pairs found in \ref{fact1} and \ref{fact2} are complete.

\begin{fact} \cite[Thm.~4.1.3]{odabasi-completeness}
\label{Odabasi}
Let $\class M$ be an abelian category as in setup \ref{setup}\footnote{In fact, as it follows from \cite{odabasi-completeness}, even less assumptions might be considered, see \cite[3.11, 3.12]{odabasi-completeness}.} and let  $(\class A,\class B)$ be a complete cotorsion pair in $\class M$. Then the following hold:
\begin{itemize}
\item[-] If $Q$ is a left rooted quiver, then the induced cotorsion pair $(\Phi(\class A), \rep_{Q}\class B)$ in $\rep_{Q}\class A$ (which exists by \ref{fact1}) is complete. 
\item[-] If $Q$ is a right rooted quiver, then the induced cotorsion pair $(\rep_{Q}\class A,\Psi(\class B))$ in $\rep_{Q}\class A$ (which exists by \ref{fact2}) is complete. 
\end{itemize} 
\end{fact}






\section{Abelian model structures on categories of quiver representations}

Based on the results stated in the previous section, we describe here a general recipe in order to produce abelian model structures on the category $\rep_{Q}\class M$, where $\class M$ is as in the setup \ref{setup}, $Q$ is left rooted  and $(\class C, \class W,\class F)$ is a hereditary Hovey triple on the ``ground category'' $\class M$. The associated complete hereditary cotorsion pairs in $\class M$ are $(\class C\cap\class W,\class F)$ and $(\class C,\class W\cap\class F)$. Using \ref{fact1} and \ref{Odabasi} we obtain two hereditary and complete cotorsion pairs in $\rep_{Q}\class M$,

\begin{equation}
(\widetilde{\class Q}=\Phi(\class C\cap\class W),\class R= \rep_{Q}\class F)\,\,\,\,\,\,  \mbox{and}\,\,\,\,\,\,    (\class Q=\Phi(\class C), \widetilde{\class R}=\rep_{Q}(\class F\cap\class W)) .
\end{equation}


We want to check if the above cotorsion pairs induce an abelian model structure on $\rep_{Q}\class M$. The following result of \cite{Gillespie-how-to} gives conditions on two complete cotorsion pairs in an abelian category $\class A$ in order for them to constitute a Hovey triple.

\begin{fact}\cite[Thm. 1.1]{Gillespie-how-to}
\label{fact-gil-how-to}
Let $\class M$ be an abelian category and assume that $(\widetilde{\class Q},\class R)$ and $(\class Q,\widetilde{\class R})$ are two hereditary complete cotorsion pairs on $\class M$ such that
\begin{itemize}
\item[(i)] $\widetilde{\class R}\subseteq\class R$\,  and\,  $\widetilde{\class Q}\subseteq\class Q$.
\item[(ii)] $\widetilde{\class R}\cap\class Q=\widetilde{\class Q}\cap\class R$.
\end{itemize}
Then there is a unique thick class $\class T$ for which $(\class Q,\class T,\class R)$ is a Hovey triple. Moreover, this class can be described as follows:
\begin{eqnarray}
\class T &=&\{X\in\mathcal M\, |\, \mbox{there exists a s.e.s.}\,\,\,  X\rightarrowtail R\twoheadrightarrow Q\,\,\,  \mbox{with}\,\,\,  R\in\widetilde{\class R},\, Q\in\widetilde{\class Q}\} \nonumber \\
&=&\{X\in\mathcal M\, |\, \mbox{there exists a s.e.s.}\,\,\,  R'\rightarrowtail Q'\twoheadrightarrow X\,\,\,  \mbox{with}\,\,\,  R'\in\widetilde{\class R},\, Q'\in\widetilde{\class Q}\}. \nonumber
\end{eqnarray}
\end{fact}

For the cotorsion pairs in $(1)$, the only nontrivial relation is $\widetilde{\class R}\cap\class Q\subseteq\widetilde{\class Q}\cap\class R$. If $X\in\widetilde{\class R}\cap\class Q$, there is a short exact sequence 
\[0\rightarrow\bigoplus_{\alpha:j\rightarrow i}X(j)\xrightarrow{\phi_{i}^{X}} X(i)\rightarrow\coker\phi_{i}^{X}\rightarrow 0\]
where $\coker\phi_{i}^{X}\in\class C$ and $X(i)\in\class C\cap\class F \cap\class W$. Note that $X\in\class R=\rep_{Q}\class F$ trivially and that $X\in\widetilde{\class Q}=\Phi(\class C\cap\class W)$ if and only if $\coker\phi_{i}^{X}\in\class W$. Since $\class W$ is a thick subcategory of $\class M$, by the short exact sequence above, we see that $\coker\phi_{i}^{X}\in\class W$ if $\class W$ is closed under (small) coproducts. We point out that this condition will be automatically satisfied for all finite and also many infinite quivers. For example, for quivers $Q$ such that for all $i\in Q_{0}$, the set $\{s(\alpha)\,|\, \alpha\in Q_{1}\,\, \mbox{with}\,\,\, t(\alpha)=i\}$ is finite. 

The above discussion proves the following:

\begin{proposition}
\label{model-phi}
Let $\class M$ be an abelian category as in setup \ref{setup} with a hereditary Hovey triple $(\class C,\class W,\class F)$ and let $Q$ be a left rooted quiver.
Then if $\class W$ is closed under (small) coproducts, 
there exists an induced hereditary Hovey triple $(\Phi(\class C),\class T,\rep_{Q}\class F)$ on the category of representations $\rep_{Q}\class M$, where $\class T$ is defined as in \ref{fact-gil-how-to}. In particular, for all $i\in Q_{0}$ and $X\in\class T$ we have $X(i)\in\class W$.
\end{proposition}

Using duals of the above arguments we easily obtain the following:

\begin{proposition}
\label{model-psi}
Let $\class M$ be an abelian category as in setup \ref{setup} with a hereditary Hovey triple $(\class C,\class W,\class F)$ and let $Q$ be a right rooted quiver.
Then if $\class W$ is closed under (small) products, 
there exists an induced hereditary Hovey triple $(\rep_{Q}\class C,\class T,\Psi(\class F))$ on the category of representations $\rep_{Q}\class M$, where $\class T$ is defined as in \ref{fact-gil-how-to}. In particular, for all $i\in Q_{0}$ and $X\in\class T$ we have $X(i)\in\class W$.
\end{proposition}

In the model structures \ref{model-phi} and \ref{model-psi}, the class of trivial objects $\class T$ is contained in the class of ``vertexwise trivial'' representations, $\rep_{Q}\class W$. For computational purposes we are interested in knowing when these two classes coincide. For this we will restrict to more special types of model structures (although still abundant).

A priori one needs two suitable complete cotorsion pairs in an abelian category $\class M$ in order to define an abelian model structure on $\class M$ (as we recalled in \ref{h1}/\ref{h2}). However, it is possible to obtain quite naturally a model structure starting with only one cotorsion pair. We recall the following from \cite{gillespie-recollement}.

\begin{definition} \cite[Dfn.~3.4]{gillespie-recollement}
Let $\class M$ be an abelian category with enough projectives. A complete cotorsion pair $(\class C,\class W)$ in $\class M$ is called \textit{projective} if $\class C\cap\class W=\Proj\class M$ and $\class W$ is thick. In this case $(\class C,\class W,\class M)$ is a Hovey triple on $\class M$. Dually, if $\class M$ has enough injectives, a complete cotorsion pair $(\class W,\class F)$ in $\class M$ is called \textit{injective} if $\class W\cap\class F=\Inj\class M$ and $\class W$ is thick. In this case $(\class M,\class W,\class F)$ is a Hovey triple on $\class M$. 

In this paper, we call an abelian model structure projective (resp. injective) if its associated Hovey triple is induced by a projective (resp. injective) cotorsion pair.
\end{definition}
 
The following two results provide us with a large class of projective (resp. injective) model structures for categories of quiver representations.

\begin{theorem}
\label{Thm-proj-model}
Let $\class M$ be an abelian category as in setup \ref{setup} and let $Q$ be a left rooted quiver. Let $(\class C,\class W)$ be a hereditary projective cotorsion pair in $\class M$ with $\class W$ closed under (small) coproducts. Then there exists a \textbf{projective model structure} on the category of representations $\rep_{Q}\class M$ with hereditary Hovey triple \[(\Phi(\class C),\rep_{Q}\class W,\rep_{Q}\class M).\]
\end{theorem}

\begin{proof} 
By assumption, $(\class C,\class W)$ is a hereditary cotorsion pair, thus from \ref{model-phi} we know that $(\Phi(\class C),\class T,\rep_Q\class M)$ is a hereditary Hovey triple on $\class M$. In particular $(\Phi(\class C),\class T)$ is a cotorsion pair in $\class M$. Moreover, from \ref{fact1} we have that $(\Phi(\class C),\rep_{Q}\class W)$ is a cotorsion pair in $\class M$. Hence we have $\rep_{Q}\class W=\class T$. The cotorsion pair obtained is projective since \[\Phi(\class C)\cap\class W=\Phi(\class C\cap\class W)=\Phi(\Proj\class M)=\Proj(\rep_{Q}\class M),\]
where the last equality is given by \ref{fact1}.
\end{proof}

Dually, we have the following:

\begin{theorem}
\label{Thm-inj-model}
Let $\class M$ be an abelian category as in setup \ref{setup} and let $Q$ be a right rooted quiver. Let $(\class W,\class F)$ be an hereditary injective cotorsion pair in $\class M$ with $\class W$ closed under (small) products. Then there exists an \textbf{injective model structure} on the category of representations $\rep_{Q}\class M$ with hereditary Hovey triple \[(\rep_{Q}\class M,\rep_{Q}\class W,\Psi(\class F)).\]
\end{theorem}

\begin{proof}
The proof of this follows the same lines as the proof of \ref{Thm-proj-model}, where one instead makes use of Proposition \ref{model-psi} in order to obtain a hereditary Hovey triple $(\rep_{Q}\class M,\class T,\Psi(\class F))$, and then argues that $\class T=\rep_{Q}\class W$ and that the cotorsion pair $(\class W,\Psi(\class F))$ is injective.
\end{proof}

\begin{remark}
We should explain how Theorems \ref{Thm-proj-model} and \ref{Thm-inj-model} connect with some classical results from the theory of model categories. Given a small category $Q$ and a cofibrantly generated model category $\class M$, it is well known that there exists an induced cofibrantly generated model structure on the functor category $\class M^{Q}$, see for instance \cite[Thm.~11.6.1]{MR1944041}. 
In the language of this paper, \cite[Thm.~11.6.1]{MR1944041} when restricted to abelian model structures says the following: 

\textit{Let $Q$ be a small category and let $\class M$ be an abelian model category with Hovey triple $(\class C,\class W,\class F)$, where the associated cotorsion pairs $(\class C,\class W\cap\class F)$ and $(\class C\cap\class W,\class F)$ are each  generated by a set (and so $\class M$ is cofibrantly generated by \cite[Lemma~6.7]{hovey}). Say $(\class C,\class W\cap\class F)=(\leftperp{(\rightperp{\class S})},\rightperp{\class S})$ for a set $\class S\subseteq \class C$. Then there exists an abelian model structure on the functor category $\class M^{Q}:=\rep_{Q}\class M$ with Hovey triple 
\[(\Sigma,\class \rep_{Q}\class W,\rep_{Q}\class F),\]
where the cotorsion pair $(\Sigma,\class \rep_{Q}(\class W \cap \class F))$ is generated by the set \[f_{*}(\class S):=\{f_{i}S\, |\, i\in Q_{0}\, , S\in\class S\}.\]}

We point out that this theorem agrees, for certain cotorsion pairs, with results of Holm and J\o rgensen, cf.~\cite[Thm.~7.4(a)]{Holm-Jorgensen2016}. Moreover, note that the class $\Sigma$ is the left hand side of a complete cotorsion pair which is generated by the set $f_{*}(\class S)$. Hence it consists of summands of transfinite extensions of objects in $f_{*}(\class S)$, \cite[3.2]{GobelTrlifaj}. For this reason it is not very computable in general. Note that \ref{Thm-proj-model} identifies this class with $\Phi(\class C)$ in case the given model on $\class M$ is projective and $Q$ is a left rooted quiver.
\end{remark}

Next, we provide some examples modelling stable categories of Gorenstein projective and Gorenstein injective representations of quivers. For a definition of these classes, we refer for instance to \cite{HHl04a}. If $R$ is a ring, $\GProj(R)$ (resp. $\GInj(R)$) denotes the class of Gorenstein projective (resp. Gorenstein injective) right $R$-modules.

\begin{example}
\label{ex-Gproj}
Let $R$ be a Noetherian commutative ring with a dualizing complex or a left-coherent and right-noetherian $k$-algebra (with $k$ a field) admitting a dualizing complex (in the sense of \cite[Setup 1.4']{PJr07}). Then from \cite[Thm.~1.10]{PJr07} the pair $(\GProj(R),\GProj(R)^{\bot})$ is hereditary projective cotorsion pair\footnote{It is also proved in \cite{MR3690524} that the pair $(\GProj(R),\GProj(R)^{\bot})$ is hereditary projective cotorsion pair over right coherent and left n-perfect rings.}. Let $Q$ be a left rooted quiver. Assuming that $Q$ is such that for all $i\in Q_{0}$ the set $\{s(\alpha)\,|\, \alpha\in Q_{1}\,\, \mbox{with}\,\,\, t(\alpha)=i\}$ is finite or assuming that $\GProj(R)^{\bot}$ is closed under coproducts\footnote{This holds for example if $R$ is Iwanaga-Gorenstein, since in this case $\GProj(R)^{\bot}$ is the class of modules of finite projective dimension \cite[Thm.~2.20]{HHl04a}.} from Theorem \ref{Thm-proj-model} we obtain a hereditary Hovey triple \[(\Phi(\GProj(R)),\rep_{Q}(\GProj(R)^{\bot}),\rep_{Q}(R)),\]
on the category of quiver representations of right $R$-modules, $\rep_{Q}(R)$.

From \cite[Thm~3.5.1]{EHS-total-acyclicity-quivers} we have that $\Phi(\GProj(R))=\GProj(\rep_{Q}(R))$, thus the above Hovey triple is
\[(\GProj(\rep_{Q}(R),\rep_{Q}(\GProj(R)^{\bot}),\rep_{Q}(R)).\]
The homotopy category of this model category is \[\Ho(\rep_{Q}(R))
\cong\sGProj(\rep_{Q}(R)),\]
the \textit{stable category of Gorenstein projective representations}.
\end{example}

\begin{example}
\label{ex-GInj}
Let $R$ be a right Noetherian ring. Then the pair $(^{\bot}\GInj(R),\GInj(R))$ is a hereditary injective cotorsion pair \cite[7.3]{HKr05}\footnote{In the recent work \cite{Jans-GInj} the authors prove that $(^{\bot}\GInj(R),\GInj(R))$ is a hereditary injective cotorsion pair over any ring.}. Assuming that $Q$ is a quiver such that for all $i\in Q_{0}$ the set $\{t(\alpha)\,|\, \alpha\in Q_{1}\,\, \mbox{with}\,\,\, s(\alpha)=i\}$ is finite or assuming that $^{\bot}\GInj(R)$ is closed under products\footnote{Again, this holds if $R$ is Iwanaga-Gorenstein, since in this case $^{\bot}\GInj(R)$ is the class of modules of finite injective dimension \cite[Thm.~2.22]{HHl04a}.}, then by Theorem \ref{Thm-inj-model} we obtain a hereditary Hovey triple \[(\rep_{Q}(R),\rep_{Q}(^{\bot}\GInj(R)),\Psi(\GInj(R))),\]
on the category of quiver representations of right $R$-modules, $\rep_{Q}(R)$.

From \cite[Thm~3.5.1]{EHS-total-acyclicity-quivers} we have that $\Psi(\GInj(R))=\GInj(\rep_{Q}(R))$, thus the above Hovey triple is
\[(\rep_{Q}(R),\rep_{Q}(^{\bot}\GInj(R)),\GInj(\rep_{Q}(R))).\]
The homotopy category of this model category is \[\Ho(\rep_{Q}(R))\cong\sGInj(\rep_{Q}(R)),\]
the \textit{stable category of Gorenstein injective representations}. 
\end{example}

\section{Quiver representations over Ding-Chen rings}
The examples \ref{ex-Gproj} and \ref{ex-GInj} admit generalizations which are worth mentioning. Gillespie in \cite{MR2607410} based on work of Ding and Chen \cite{MR1396867} defines Ding-Chen rings as a generalization of Gorenstein rings. 
A ring is called Ding-Chen if it is left and right coherent with $\FPI-\mathrm{dim}_{R}R$ and $\FPI-\mathrm{dim}R_{R}$ both finite\footnote{Here $\FPI-\mathrm{dim}$ denotes the fp-injective dimension. We recall that an $R$-module $M$ is called fp-injective if for any finitely presented module $F$ we have $\Ext_{R}^{1}(F,M)=0,$  see \cite{Stenstrom-FPI}. These modules define a (relative) homological dimension, see \cite[Ch.~8]{rha}.}. In this case from \cite{MR1202159} we necessarily have $\FPI-\mathrm{dim}_{R}R=n=\FPI-\mathrm{dim}R_{R}$ for some $n\in\mathbb{N}$. Note that if $R$ is two-sided Noetherian then this definition recovers the Iwanaga-Gorenstein rings.

Gillespie studies in \cite{MR2607410} Ding projective, injective and flat modules which stand for generalizations of Gorenstein projective, injective and flat modules respectively.

\begin{remark}
The definition of Ding projective and Ding injective modules over a ring involves the concepts of flat and fp-injective modules respectively \cite[Dfn.~3.2/3.7]{MR2607410}. Since we are interested in Ding projective and  Ding injective representations of quivers, we need to make sense of flatness and fp-injectivity in a more general context than module categories. The appropriate setup to define such notions is that of a locally finitely presented additive (usually Grothendieck) category, see \cite{AdamekRosicky, WCB94}. In this context an object $M$ is called flat if any epimorphism with target $M$ is pure, and dually, $M$ is called fp-injective if any monomorphism with source $M$ is pure. For a locally finitely presented Grothendieck category $\class A$ and a quiver $Q$, the category of quiver representations $\rep_{Q}\class A$ is again locally finitely presented Grothendieck \cite[Cor.~1.54]{AdamekRosicky}.
\end{remark}

\begin{definition}
Let $\class A$ be a locally finitely presented Grothendieck category. An object $M$ in $\class A$ is called \textit{Ding projective} if there exists an exact complex of projective objects in $\class A$ which has $M$ as a syzygy and remains exact after applying functors of the form $\Hom_{\class A}(-,F)$, for $F$ a flat object in $\class A$. We denote the class of Ding projective objects in $\class A$ by $\DProj(\class A)$.
\end{definition}

\begin{definition}
Let $\class A$ be a locally finitely presented Grothendieck category. An object $M$ in $\class A$ is called \textit{Ding injective} if there exists an exact complex of injectives in $\class A$ which has $M$ as a syzygy and remains exact after applying functors of the form $\Hom_{\class A}(F,-)$, for $F$ an fp-injective object in $\class A$. We denote the class of Ding injective objects in $\class A$ by $\DInj(\class A)$.
\end{definition}

We will make use of the following facts which concern flat and fp-injective representations, i.e. flat and fp-injective objects in the category of quiver representations of right $R$-modules, $\rep_{Q}(R)$.

\begin{fact}
\label{korean}
Let $R$ be a ring, $Q$ a quiver and let $X\in\rep_{Q}(R)$. Then we have:
\begin{itemize}
\item[(i)] \cite[3.4/3.7]{MR2100360} If $X$ is a flat representation, then for each vertex $v\in Q_{0}$, the natural map $\phi_{v}^{X}$ as in \ref{quiver} is a pure monomorphism and $X(v)$ is flat. The converse holds in case $Q$ is left rooted.
\item[(ii)] \cite[4.4/4.10]{MR3275710} If $X$ is an fp-injective representation, then for each vertex $v\in Q_{0}$, the natural map $\psi_{v}^{X}$ as in \ref{quiver} is a pure epimorphism and $X(v)$ is fp-injective. The converse holds in case $Q$ is right rooted and $R$ is right coherent.
\end{itemize}
\end{fact}

The proofs of the following two results are based on techniques developed in Eshraghi et al.~\cite{EHS-total-acyclicity-quivers}, although some modifications are needed. We keep the presentation as concise as possible.

\begin{lemma}  \textnormal{(cf.~\cite[3.1.5]{EHS-total-acyclicity-quivers})}
\label{tech}
Let $R$ be a ring, $Q$ a quiver and let $X\in\rep_{Q}R$. Then the following hold:
\begin{itemize}
\item[(i)] Assuming that $Q$ is left rooted, if for all $v\in Q_{0}$, $X(v)$ is flat, then $\mathrm{Flat-dim}(X)\leq 1$.
\item[(ii)] Assuming that $R$ is right coherent and $Q$ is right rooted, if for all $v\in Q_{0}$, $X(v)$ is fp-injective, then $\FPI-\mathrm{dim}(X)\leq 1$.
\end{itemize}
\end{lemma}

\begin{proof}
(ii)  Keeping the notation as in \ref{adjoints}, from \cite[3.1(1)]{EHS-total-acyclicity-quivers} there exists a short exact sequence 
\[0\rightarrow X\rightarrow \prod\limits_{v\in Q_{0}}g_{v}X(v)\rightarrow \prod\limits_{\alpha\in Q_{1}}g_{s(\alpha)}X(t(\alpha))\rightarrow 0\]
 in the category $\rep_{Q}(R)$. 
For the representation $g_{v}X(v)$ and for all $w\in Q_{0}$, the natural map $\psi_{g_{v}X(v)}^{w}$, is a split epimorphism. Moreover, by assumption, vertexwise the representation $g_{v}X(v)$ consists of fp-injective modules, thus by \ref{korean}(ii) we obtain that $g_{v}X(v)$ is an fp-injective representation. Hence the middle term in the above short exact sequence is an fp-injective representation (since the class of fp-injective representations is closed under products \cite[App.~B]{Stoviceck-on-purity}). To prove that the term on the right hand side is fp-injective, in order to simplify the notation, denote $Y:=\prod_{v\in Q_{0}}g_{v}X(v)$ and $W:=\prod_{\alpha\in Q_{1}}g_{s(\alpha)}X(t(\alpha))$, so the displayed short exact sequence above is $0\rightarrow X\rightarrow Y\rightarrow W\rightarrow 0$.

Consider for each vertex $v\in Q_{0}$ the commutative diagram of $R$-modules

\begin{displaymath}
  \xymatrix@C=3pc{
    Y(v) \ \ar@{->>}[r] \ar[d]_-{\psi_{v}^{Y}} & 
    W(v)  \ar[d]^-{\psi_{v}^{W}} \\
    \prod\limits_{\alpha:v\rightarrow t(\alpha)}Y(t(\alpha)) \ \ar@{->>}[r] & \prod\limits_{\alpha:v\rightarrow t(\alpha)}W(t(\alpha)).
  }
\end{displaymath}
Then observe that the top map is a pure epimorphism (by the assumption that its kernel, which is $X(v)$, is an fp-injective module), hence also the bottom map is a pure epimorphism. Moreover, the map on the left hand side is a split epimorphism, hence $\psi_{v}^{W}$ is a pure epimorphism. Since $W$ vertexwise consists of fp-injectives, in view of \ref{korean}(ii) we obtain the desired result.
 

The proof of (i) is dual where one uses instead a short exact sequence of representations ending in $X$, see \cite[3.1(2)]{EHS-total-acyclicity-quivers}, and makes use of \ref{korean}(i).
\end{proof}

\begin{proposition}\textnormal{(cf.~\cite[3.5.1]{EHS-total-acyclicity-quivers})}
\label{phi-psi-lemma}
Let $R$ be a ring and let $Q$ be a quiver. Then the following hold:
\begin{itemize}
\item[(i)] If $Q$ is left rooted then $\Phi(\DProj(R))=\DProj(\rep_{Q}(R))$.
\item[(ii)] If $R$ is right coherent and $Q$ is right rooted then $\Psi(\DInj(R))=\DInj(\rep_{Q}(R))$.
\end{itemize}
\end{proposition}

\begin{proof}
(ii) Assume that $D\in\DInj(\rep_{Q}(R))$, which by definition means that there exists an exact complex of injective representations,
\[X^{\bullet} = \cdots\rightarrow X^{-1}\rightarrow X^{0}\rightarrow X^{1}\rightarrow\cdots,\]
which has\, $D=\ker(X^{0}\rightarrow X^{1}):=Z^{0}(X^{\bullet})$\, and remains exact after applying functors of the form $\Hom_{\rep_{Q}(R)}(\FPI,-)$. We need to prove that $D\in\Psi(\DInj(R))$.

We first prove that for each vertex $v\in Q_{0}$ we have $D(v)\in\DInj(R)$. Indeed, the complex of $R$-modules $X^{\bullet}(v)$ is exact, consists of injective modules \cite[2.1]{EEGR-injective-quivers}, has $D(v)$ as a syzygy,  and it remains to check that, for any fp-injective module $F$, the complex $\Hom_{R}(F,X^{\bullet}(v))$ is exact.
Now, the functor $f_{v}M$, as in \ref{adjoints}, from \cite[5.2(a)]{Holm-Jorgensen2016} is such that the complex $\Hom_{R}(F,X^{\bullet}(v))$ is exact if and only if the complex of representations\, $\Hom_{\rep_{Q}(R)}(f_{v}F,X^{\bullet})$ is exact. Note that $f_{v}F$ vertexwise consists of fp-injectives (since $\FPI(R)$ is closed under coproducts \cite{Stenstrom-FPI}), hence Lemma \ref{tech}(ii) implies that $f_{v}F$ is a representation with fp-injective dimension $\leq 1$. This means that there exists a short exact sequence of representations $0\rightarrow f_{v}F\rightarrow Y_{0}\rightarrow Y_{1}\rightarrow 0$\, with $Y_{0}, Y_{1}$\, fp-injectives. Hence we obtain a short exact sequence of complexes of representations
\[0\rightarrow\Hom (Y_{1},X^{\bullet})\rightarrow\Hom(Y_{0},X^{\bullet})\rightarrow\Hom(f_{v}F,X^{\bullet})\rightarrow 0,\]
where by the assumption on $X$ the two leftmost terms are exact. Hence so is $\Hom (f_{v}F,X^{\bullet})$. Thus $D(v)\in\DInj(R)$.

Next, we show that for all $v\in Q_{0}$ the natural map $\psi^{D}_{v}$ of \ref{quiver} is an epimorphism with kernel in $\DInj(R)$. 
For this fix a vertex $v\in Q_{0}$, then in the commutative diagram 
\begin{displaymath}
  \xymatrix@C=3pc{
    X^{0}(v) \ \ar@{->>}[r]^-{\partial^{0}_{v}} \ar[d]_-{\psi_{v}^{X^{0}}} & 
    D(v)  \ar[d]^-{\psi_{v}^{D}} \\
    \prod\limits_{\alpha:v\rightarrow t(\alpha)}X^{0}(t(\alpha)) \ \ar@{->>}[r] & \prod\limits_{\alpha:v\rightarrow t(\alpha)}D(t(\alpha)),  
  }
\end{displaymath}
the map on the left is an epimorphism (by the characterization of injective representations in \cite[2.1]{EEGR-injective-quivers}), hence $\psi_{D}^{v}$ is also an epimorphism. Consider the degreewise split short exact sequence of complexes of $R$-modules
\[0\rightarrow\ker\psi^{v}_{X^{\bullet}}\rightarrow X^{\bullet}(v)\rightarrow\prod\limits_{\alpha:v\rightarrow t(\alpha)}X^{\bullet}(t(\alpha))\rightarrow 0.\]
By a two-out-of-three argument we see that the complex on the left is an exact complex of injectives which stays exact after applying functors of the form $\Hom_{R}(\FPI(R),-)$, and moreover has $\ker(\psi_{D}^{v})$ as a syzygy. Thus $\ker(\psi_{D}^{v})\in\DInj(R)$ which concludes the proof that $D\in\Psi(\DInj(R))$.

Conversely, let $D\in\Psi(\DInj(R))$. We want to prove that $D\in\DInj(\rep_{Q}(R))$, i.e. we want to find an exact complex of injective representations $E^{\bullet}$ which has $D$ as a syzygy and remains exact after applying functors of the form $\Hom(\FPI,-)$. Recall the transfinite sequence $(W_{\lambda})$ as defined in \ref{quiver}. Following the proof of \cite[Thm.~3.5.1(a)]{EHS-total-acyclicity-quivers}, we will see how to construct recursively, for each ordinal $\lambda$, an exact complex $E^{\bullet}_{\lambda}$ of injective representations of the subquiver $Q^{\lambda}:=\{v\in Q_{0}\,\,|\,\,v\in W_{\lambda}\}$, which is such that for all $v\in Q^{\lambda}$, the complex $E^{\bullet}_{\lambda}(v)$\, is\, $\Hom(\FPI(R),-)$--exact and has $D(v)$ as a syzygy. We give the first step of this construction:

 For each vertex $v\in W_{1}:=\{i\in Q_{0}\, |\, i\,\,\mbox{is not the target of any arrow in Q}\}$, by the assumption that $D(v)\in\DInj(R)$, there exists an exact complex of injectives $I_{v}^{\bullet}$ with $Z^{0}(I_{v}^{\bullet})=D(v)$. Thus we may use the right adjoint $g_{v}$ from \ref{adjoints} to obtain an exact complex $E_{1}^{\bullet}:=g_{v}I^{\bullet}_{v}$ of injective representations of the subquiver $Q^{1}:=\{v\in Q_{0}\,\,|\,\,v\in W_{1}\}$. Then note that for all $v\in Q_{0}$, the exact complex of injectives $E_{1}^{\bullet}(v)$ has $D(v)$ as a syzygy and is $\Hom(\FPI(R),-)$--exact.
 
Then one can follow verbatim the rest of the argument in \cite[Thm.~3.5.1(a)]{EHS-total-acyclicity-quivers} to obtain an exact complex $E^{\bullet}=\cup_{\lambda}E^{\bullet}_{\lambda}$ of injective representations which has $D$ as a syzygy, and is such that for all $v\in Q_{0}$ the exact complex of injectives $E^{\bullet}(v)$ is $\Hom_{R}(\FPI(R),-)$--exact. The proof will be finished if we show that for any fp-injective representation $F$, the complex $\Hom(F,E^{\bullet})$ is exact. This follows after considering the degreewise split short exact sequence 
\[0\rightarrow E^{\bullet}\rightarrow \prod\limits_{v\in Q_{0}}g_{v}E^{\bullet}(v)\rightarrow \prod\limits_{\alpha\in Q_{1}}g_{s(\alpha)}E^{\bullet}(t(\alpha))\rightarrow 0\]
and observing that, for any fp-injective representation $F$, the two rightmost terms are $\Hom(F,-)$--exact.\\
The proof of (i) is completely dual, one just needs to make use of the duals of the arguments in the proof of (i), which are provided in our previous results. 
\end{proof}

We now give a projective model structure for Ding projective representations over a Ding-Chen ring. We will need the following result of Ding and Chen.

\begin{fact}\cite[Prop.~3.16]{MR1202159}
\label{finitistic}
Let $R$ be a Ding-Chen ring with $\FPI-\mathrm{dim}_{R}R\leq n$ and $\FPI-\mathrm{dim}R_{R}\leq n$ for some integer $n\in\mathbb{N}$. Then for any right $R$-module $M$ we have bi--implications:
\begin{equation*}
\begin{split}
\mathrm{Flat-dim}(M)<\infty & \Leftrightarrow \mathrm{Flat-dim}(M)\leq n\\
 &\Leftrightarrow  \FPI-\mathrm{dim}(M)<\infty\\
 &\Leftrightarrow  \FPI-\mathrm{dim}(M)\leq n.
\end{split}
\end{equation*}
\end{fact}

\begin{theorem}
\label{ding-proj-model}
Let $R$ be a Ding-Chen ring and let $Q$ be a left rooted quiver. Then there exists a hereditary Hovey triple $(\DProj(\rep_{Q}\class (R)),\rep_{Q}\class W,\rep_{Q}(R))$ on the category of quiver representations of right $R$-modules, $\rep_{Q}(R)$. The homotopy category of this model structure is \[\Ho(\rep_{Q}(R))\cong\sDProj(\rep_{Q}(R)),\]
the \textit{stable category of Ding projective representations}.
\end{theorem}

\begin{proof}
From \cite[Thm.~4.7]{MR2607410} we have that over a Ding-Chen ring $R$ there exists a hereditary complete cotorsion pair $(\DProj(R),\class W)$ in the category of right $R$-modules, where $\class W\cap\DProj(R)=\Proj(R)$. Moreover, from \cite[Thm~4.2]{MR2607410} the class $\class W$ consists of all modules that have finite flat dimension, hence $\class W$ is closed under coproducts by \ref{finitistic}. Hence, for a left rooted quiver $Q$, Theorem \ref{Thm-proj-model} provides us with a hereditary Hovey triple 
\[ (\Phi(\DProj(R)),\rep_{Q}\class W,\rep_{Q}(R)),\]
in the category of quiver representation of right $R$-modules, $\rep_{Q}\class(R)$. Now the result follows after applying \ref{phi-psi-lemma}(i). 
\end{proof}

Our last result concerns Ding injective representations over a Ding-Chen ring. 

\begin{theorem}
\label{dinj-model}
Let $R$ be a Ding-Chen ring and $Q$ a right rooted quiver. Then there exists a hereditary Hovey triple $(\rep_{Q}(R),\rep_{Q}\class W,\class\DInj(\rep_{Q}(R)))$ on the category of quiver representations of right $R$-modules, $\rep_{Q}(R)$. The homotopy category of this model structure is \[\Ho(\rep_{Q}(R))\cong\sDInj(\rep_{Q}(R)),\]
the \textit{stable category of Ding injective representations}. 
\end{theorem}

\begin{proof}
Dual to that of \ref{ding-proj-model} where one instead makes use of the hereditary injective cotorsion pair $(\class W,\DInj(R))$ of \cite[Thm.~4.7]{MR2607410}, where $\class W$ consists of all modules of finite fp-injective dimension \cite[Thm~4.2]{MR2607410}. Note that by \ref{finitistic} $\class W$ is closed under products. Then one applies Theorem \ref{Thm-inj-model} and \ref{phi-psi-lemma}(ii).
\end{proof}

\section*{Acknowledgement}
The author would like to thank his PhD advisors, Sergio Estrada and Henrik Holm, for providing comments on a draft version of this manuscript.

\bibliographystyle{amsplain}
\bibliography{models-on-quiver-reps.bib}

\end{document}